\documentclass{amsart}

\usepackage{amsrefs,amssymb,hyperref,tikz-cd}

\hypersetup{
    colorlinks=true,
    linkcolor=blue,
    urlcolor=teal,
    citecolor=magenta
}

\newtheorem{proposition}{Proposition}
\newtheorem{corollary}{Corollary}
\theoremstyle{definition}
\newtheorem{definition}{Definition}
\theoremstyle{remark}
\newtheorem{remark}{Remark}
\newtheorem{example}{Example}

\renewcommand\le\leqslant
\renewcommand\ge\geqslant
\newcommand\To\Rightarrow
\newcommand\id{{\mathrm{id}}}
\newcommand\ph\varphi
\newcommand\opp{^{\mathrm{op}}}
\newcommand\1{^{-1}}
\newcommand\nerve{{N_\bullet}}
\newcommand\BbbZ{{\mathbb Z}}
\newcommand\BbbC{{\mathbb C}}
\newcommand\mbfDelta{{\mathbf\Delta}}
\newcommand\mscrA{{\mathcal A}}
\newcommand\mscrC{{\mathcal C}}
\newcommand\mscrE{{\mathcal E}}
\newcommand\mscrZ{{\mathcal Z}}
\newcommand\BC{{\mathrm{BC}}}
\newcommand\Kan{{\mathrm{Kan}}}
\newcommand\nsubset{{\not\subset}}

\newcommand{\romb}{\tikz[baseline=-0.5ex]{
  \draw[scale=0.1,thick] (0,1) -- (1,0) -- (0,-1) -- (-1,0) -- cycle;
  \draw[scale=0.1,thick] (-1.1,0) -- (1.1,0);
}}

\begin{document}

\title{Beck-Chevalley Conditions in Simplicial Sets}
\author{Gaga Chakhvashvili}
\thanks{Based on the author's Master thesis, Tbilisi State University, July 2025.\\This work was supported by Shota Rustaveli National Science Foundation of Georgia (SRNSFG), grant FR-24-9660.}

\begin{abstract}
We describe certain class of simplicial sets introduced by Dmitry Skvortsov and Valentin Shehtman; we call such simplicial sets Skvortsov-Shehtman complexes. An example of a Skvortsov-Shehtman complex that is not a Kan complexes is given.
\end{abstract}

\maketitle

\section*{Introduction}

A few years ago, Dmitry Skvortsov and Valentin Shehtman associated with every first order theory a simplicial set, with $n$-simplices corresponding to $n$-types of the theory. For their research, it was important that this simplicial set satisfies certain conditions called by them the Beck-Chevalley conditions \cite{Shehtman} (see below). We call simplicial sets with this property \emph{Skvortshov-Shehtman complexes}, or SS-complexes for short.

Every Kan complex is an SS-complex. Our aim is to show that the converse is not true, i.~e. there exist SS-complexes that are not Kan complexes.

Every SS-complex satisfies the Kan conditions in dimensions $\le1$. In particular, the nerve $\nerve\BbbC$ of a small category $\BbbC$ is an SS-complex iff it is a Kan complex, i.~e. iff $\BbbC$ is a groupoid.

To find SS-complexes which are not Kan complexes we investigate the effect of Beck-Chevalley conditions on the (Duskin) nerve $\nerve\mscrC$ of a small 2-category $\mscrC$. In particular, we consider a single object single 1-morphism 2-category $\mscrZ(M)$ corresponding to a commutative monoid $M$, such that $\nerve\mscrZ(M)$ is a Kan complex iff $M$ is a group. We then describe necessary and sufficient conditions on $M$ for the nerve of $\mscrZ(M)$ to be an SS-complex.

We construct a commutative monoid $M_0$ such that $\nerve\mscrZ(M_0)$ is not an SS-complex, and another commutative monoid $M_1$ such that $\nerve\mscrZ(M_1)$ is an SS-complex but not a Kan complex.

We finish with some questions.

\section*{Beck-Chevalley conditions and Skvortsov-Shehtman fibrations}

We use the standard notation for the category $\mathbf{Set}^{\mbfDelta\opp}$ of simplicial sets, as in e.~g. \cite{GoJa} or \cite{May}. In particular, $\Delta[n]$ denotes the standard $n$-dimensional simplex $\hom_\mbfDelta(-,[0,n])$, carrier of the generic $n$-simplex $\id_{[0,n]}\in\Delta[n]_n$.

Recall that the $p$th \emph{horn} $\Lambda_p[n]\subseteq\Delta[n]$ is the union of all facets of $\Delta[n]$ except for the $p$th one, and that a simplicial map $f:E\to B$ is a \emph{Kan fibration} iff for any $0\le p\le n$, any commutative square
\[
\begin{tikzcd}
\Lambda_p[n]\ar[r]\ar[d,hook] & E \ar[d,"f"]\\
\Delta[n]\ar[r]\ar[ur,dashed] & B
\end{tikzcd}
\]
admits a filler as above, i.~e. a map $\Delta[n]\to E$ rendering both resulting triangles commutative. A \emph{Kan complex} is a simplicial set $K$ such that the unique map $K\to*$ to the terminal simplicial set $*=\Delta[0]$ is a Kan fibration, that is, the \emph{Kan conditions} $\Kan_p[n]$ are satisfied for all $n>0$ and all $0\le p\le n$: for any $n$-tuple $c_0$, ..., $c_{p-1}$, $c_{p+1}$, ..., $c_n$ of $(n-1)$-simplices of $K$ with $d_i(c_j)=d_{j-1}(c_i)$ for all $0\le i<j\le n$ with $p\notin\{i,j\}$, there exists an $n$-simplex $x$ of $K$ with $d_ix=c_i$ for all $0\le i\le n$ with $i\ne p$.

We next describe the notions introduced by Skvortsov and Shehtman.

\begin{definition}[\cite{Shehtman}]
For $0\le p<q\le n$ let the \emph{$p,q$-th rhombus} $\romb_{p,q}[n]\subseteq\Delta[n]$ be the union of the $p$th and the $q$th face of $\Delta[n]$.

Call a simplicial map $f:E\to B$ a \emph{Skvortsov-Shehtman fibration} if for all $0\le p<q\le n$, $n>1$, every commutative square
\[
\begin{tikzcd}
\romb_{p,q}[n]\ar[r]\ar[d,hook] & E \ar[d,"f"]\\
\Delta[n]\ar[r]\ar[ur,dashed] & B
\end{tikzcd}
\]
admits a filler.

Call a simplicial set $S$ a \emph{Skvortsov-Shehtman complex} (SS-complex for short) if the map $S\to*$ is a Skvortsov-Shehtman fibration.

Thus $S$ is an SS-complex iff the following \emph{Beck-Chevalley conditions} $\BC_{p,q}[n]$ are satisfied: for all $n>1$ and all $0\le p<q\le n$, for any $(n-1)$-simplices $c_p$, $c_q$ of $S$ with $d_pc_q=d_{q-1}c_p$ there exists an $n$-simplex $x$ of $S$ with $d_px=c_p$ and $d_qx=c_q$.
\end{definition}

\begin{remark}
The name ``Beck-Chevalley condition'' comes from the fact that it is equivalent to the requirement that each square of the form
\[
\begin{tikzcd}
    S_n\ar[r,"d_q"]\ar[d,"d_p"'] &S_{n-1}\ar[d,"d_p"]\\
    S_{n-1}\ar[r,"d_{q-1}"'] &S_{n-2}
\end{tikzcd}
\]
is a \emph{weak pullback} square, that is, the induced map from $S_n$ to the pullback $S_{n-1}\times_{S_{n-2}}S_{n-1}$ is onto. This condition is thus related to the conditions on bifibrations with the same name related to descent, see e.~g. \cite{Dusko}.
\end{remark}

\begin{proposition}\label{prop:kantobc}
    Every Kan fibration is a Skvortsov-Shehtman fibration; in particular, every Kan complex is an SS-complex.
\end{proposition}

\begin{proof}
It is well known that each Kan fibration has the lifting property with respect to any trivial cofibration (see e.~g. \cite{GoJa}). And the inclusions $\romb_{p,q}[n]\hookrightarrow\Delta[n]$ are clearly trivial cofibrations for $n>1$, both $\romb_{p,q}[n]$ and $\Delta[n]$ being contractible.
\end{proof}

\begin{proposition}\label{prop:upto2}
The conjunction of all Kan conditions $\Kan_p[n]$ for $n\le2$ is equivalent to the conjunction of all Beck-Chevalley conditions $\BC_{p,q}[n]$ for $n\le2$. 
\end{proposition}
\begin{proof}
Indeed the conditions $\Kan_0[1]$ and $\Kan_1[1]$ are both trivially satisfied, while the Beck-Chevalley conditions are just absent for $n=1$. For $n=2$, it is straightforward to see that $\Kan_0[2]$ is the same as $\BC_{1,2}[2]$, $\Kan_1[2]$ is the same as $\BC_{0,2}[2]$ and $\Kan_2[2]$ is the same as $\BC_{0,1}[2]$.
\end{proof}

\begin{remark}
Note that the conditions analogous to $\Kan_p[1]$ for Kan fibrations are nontrivial in general.
\end{remark}

\section*{Nerves}

\begin{corollary}
The nerve $\nerve\BbbC$ of a small category $\BbbC$ is an SS-complex if and only if it is a Kan complex.    
\end{corollary}
\begin{proof}
In one direction this follows from Proposition \ref{prop:kantobc}.

The reverse direction follows from Proposition \ref{prop:upto2}, noting that every $\nerve\BbbC$ satisfies $\BC_{0,2}[2]$, while $\nerve\BbbC$ satisfies either $\BC_{0,1}[2]$ or $\BC_{1,2}[2]$ iff $\BbbC$ is a groupoid, and then $\nerve\BbbC$ is a Kan complex.
\end{proof}

We now turn to 2-categories. Just in case, here is our notation for the horizontal and vertical composition of 2-morphisms. In
\begin{center}
\begin{tikzcd}
 C & C'
  \arrow[l, bend right=60,"f"'{name=F}] 
  \arrow[l, "g"{description, name=G}] 
  \arrow[l, bend left=60, "h"{name=H}]
  & {}
  \arrow[Rightarrow, from=F, to=G, "\alpha", shorten <=3pt, shorten >=1pt]
  \arrow[Rightarrow, from=G, to=H, "\beta", shorten >=4pt]
\end{tikzcd}\qquad 
\begin{tikzcd}
C &
C' \arrow[l, bend right=50,"f"'{name=F}] \arrow[l,bend left=50,"g"{name=G}]
&
C''
\arrow[l, bend right=50,"f'"'{name=FF}] \arrow[l,bend left=50,"g'"{name=GG}]
\arrow[Rightarrow, from=F, to=G, "\alpha", shorten <=7pt, shorten >=7pt]
\arrow[Rightarrow, from=FF, to=GG, "\alpha'", shorten <=7pt, shorten >=7pt]
\end{tikzcd} 
\end{center}
the vertical composition of $\beta$ with $\alpha$ will be denoted by $\beta\cdot\alpha$, while the horizontal composition of $\alpha$ and $\alpha'$ will be denoted by $\alpha\alpha'$; moreover for brevity horizontal compositions of the form $\id_f\alpha$ will be denoted by $f\alpha$ and those of the form $\alpha\id_f$ by $\alpha f$. Thus for example in the above diagram on the right, 
\(
\alpha\alpha'=\alpha g'\cdot f\alpha'=g\alpha'\cdot\alpha f'
\).

\begin{definition}[\cites{Street,Duskin}]\label{def:dnerve}
The \emph{Duskin nerve} of a small 2-category $\mscrC$ is the simplicial set $\nerve\mscrC$ with the $n$-simplices ($n\ge0$) given by the following data:
\begin{itemize}
\item objects $C_i$ for all $0\le i\le n$;
\item 1-morphisms $f_{ij}:C_j\to C_i$ for all $0\le i<j\le n$;
\item 2-morphisms $\alpha_{ijk}:f_{ij}f_{jk}\To f_{ik}$ for all $0\le i<j<k\le n$;
\item such that the diagrams
\[
\begin{tikzcd}
&f_{ij}f_{jk}f_{kl}
\ar[dl,Rightarrow,"\alpha_{ijk}f_{kl}"']
\ar[dr,Rightarrow,"f_{ij}\alpha_{jkl}"]\\
f_{ik}f_{kl}\ar[dr,Rightarrow,"\alpha_{ikl}"']&&f_{ij}f_{jl}\ar[dl,Rightarrow,"\alpha_{ijl}"]\\
&f_{il}
\end{tikzcd}
\]
in categories $\hom_\mscrC(C_l,C_i)$ commute for all $0\le i<j<k<l\le n$.
\end{itemize}

The $j$th face maps are given by omitting all items involving the index $j$, and the $j$th degeneracy maps by repeating $C_j$ twice and inserting identity morphisms at appropriate places, just as for nerves of 1-categories.
\end{definition}

Then
\begin{proposition}\label{prop:kanerve}
The Duskin nerve $\nerve\mscrC$ of a small 2-category $\mscrC$ is a Kan complex if and only if every 1-morphism of $\mscrC$ is invertible up to invertible 2-morphisms and every 2-morphism of $\mscrC$ is invertible.
\end{proposition}
\begin{proof}
Let us begin with 2-morphisms. Given $\alpha:f\To g$ with $f,g:C'\to C$, let $C_0=C_1=C$, $C_2=C_3=C'$, $f_{01}=\id_C$, $f_{23}=\id_{C'}$, $f_{13}=f$, $f_{12}=f_{02}=f_{03}=g$, $\alpha_{013}=\alpha$, $\alpha_{012}=\alpha_{023}=\id_g$. Then it is easy to check that these define a horn $\Lambda_0[3]\to\nerve\mscrC$ such that its filler requires a 2-morphism $\beta:g\To f$ with $\alpha\cdot\beta=\id_g$. It follows that every 2-morphism must have a right inverse, which implies that every 2-morphism is an isomorphism.

Now for 1-morphisms the conditions $\Kan_0[2]$ and $\Kan_2[2]$ imply that for any $f:C'\to C$ there is a $g:C\to C'$ and 2-morphisms $fg\To\id_C$, $gf\To\id_{C'}$.
\end{proof}

On the other hand we have
\begin{proposition}\label{prop:catbc}
    The Duskin nerve $\nerve\mscrC$ of a small 2-category $\mscrC$ satisfies the condition $\BC_{pq}[n]$ ($0\le p<q\le n$, $n>1$) iff for any objects $C_i$, $0\le i\le n$, any morphisms $f_{ij}:C_j\to C_i$, $0\le i<j\le n$ and any 2-morphisms $\alpha_{ijk}:f_{ij}f_{jk}\To f_{ik}$ for $0\le i<j<k\le n$ with $\{p,q\}\nsubset\{i,j,k\}$ such that the appropriate diagrams from Definition \ref{def:dnerve} commute, there exist 2-morphisms
\begin{itemize}
\item $\lambda_{ipq}:f_{ip}f_{pq}\To f_{iq}$, $0\le i<p$,
\item $\mu_{pjq}:f_{pj}f_{jq}\To f_{pq}$, $p<j<q$,
\item $\rho_{pqk}:f_{pq}f_{qk}\To f_{pk}$, $q<k\le n$
\end{itemize}
making the remaining diagrams of the above type commutative, i.~e. satisfying the equations
\begin{enumerate}
\item $\lambda_{ipq}\cdot\alpha_{ijp}f_{pq}=\alpha_{ijq}\cdot f_{ij}\lambda_{jpq}$ for all $0\le i<j<p$;
\item $\mu_{pkq}\cdot\alpha_{pjk}f_{kq}=\mu_{pjq}\cdot f_{pj}\alpha_{jkq}$ for all $p<j<k<q$;
\item $\alpha_{pkl}\cdot\rho_{pqk}f_{kl}=\rho_{pql}\cdot f_{pq}\alpha_{qkl}$ for all $q<k<l\le n$;
\item $\lambda_{ipq}\cdot f_{ip}\mu_{pkq}=\alpha_{ikq}\cdot\alpha_{ipk}f_{kq}$ for all $0\le i<p<k<q$;
\item $\alpha_{iql}\cdot\lambda_{ipq}f_{ql}=\alpha_{ipl}\cdot f_{ip}\rho_{pql}$ for all $0\le i<p<q<l\le n$;
\item $\rho_{pql}\cdot\mu_{pjq}f_{ql}=\alpha_{pjl}\cdot f_{pj}\alpha_{jql}$ for all $p<j<q<l\le n$.
\end{enumerate}
\end{proposition}

\section*{2-categories from monoids}

Let us consider Duskin nerves for single object 2-categories. 

Let us further restrict to the case of the single object 2-subcategories of the 2-category of small categories, that single object being itself a single object category. If we view the latter as a monoid $M$, we may describe
the resulting 2-category $\mscrE(M)$ as follows. It has a single object, its 1-morphisms are endomorphisms of $M$, and for two such endomorphisms $f$, $g$ the 2-morphisms, i.~e. natural transformations $f\To g$,
are in one-to-one correspondence with elements $a\in M$ such that the equality
\begin{equation}\label{eq:afga}\tag{$*$}
af(m)=g(m)a
\end{equation}
holds for any $m\in M$. Identity 2-morphisms are given by the unit of $M$ and vertical composition by the multiplication in $M$. As for the horizontal composition, it is given by $fa=f(a)$ and $af=a$.

Since we are interested in the Beck-Chevalley conditions, in view of Proposition \ref{prop:upto2} we further restrict to the 2-subcategory of $\mscrE(M)$ with 1-morphisms self-equivalences only. Then as it happens, this actually means to restrict further to automorphisms. Indeed we have
\begin{proposition}
Every self-equivalence of a category with single object is an automorphism.
\end{proposition}
\begin{proof}
Under our description, an endomorphism $f:M\to M$ of a monoid $M$ represents a self-equivalence of the corresponding single object category if and only if there is another endomorphism $g:M\to M$ and invertible elements $u,v\in M$ satisfying
\[
ufg(m)=mu
\]
and
\[
vgf(m)=mv
\]
for all $m\in M$.

We thus have $fg(m)=u\1mu$ and $gf(m)=v\1mv$, hence
\[
vgf(m)v\1=m=fg(umu\1)
\]
for all $m\in M$. Let then $\tilde f:M\to M$ be the endomorphism given by
\[
\tilde f(m)=vgfg(umu\1)v\1=vg(m)v\1=g(umu\1).
\]
Then
\[
\tilde ff(m)=vgf(m)v\1=m=fg(umu\1)=f\tilde f(m)
\]
so that $f$ is an automorphism with inverse $\tilde f$.
\end{proof}

Taking into account that for automorphisms $f,g$ the condition \eqref{eq:afga} only depends on $fg\1$ we thus arrive at the following
\begin{definition}\label{def:2aut}
For a monoild $M$, $\mscrA(M)$ is the single object 2-category with 1-morphisms automorphisms of $M$, and 2-morphisms $f\To g$ pairs $(a,f)$ where $a\in M$ satisfies
\[
ma=af(g\1(m))
\]
for every $m\in M$. Composition of 1-morphisms is composition of automorphisms of $M$, vertical composition of 2-morphisms is given by $(b,g)\cdot(a,f)=(ba,f)$ and horizontal composition is determined by $f'(a,f)=(f'(a),f'f)$ and $(a,f)f'=(a,ff')$.
\end{definition}

We are thus led naturally to the question what conditions on $M$ are equivalent to $\nerve\mscrA(M)$ being an SS-complex.

As it happens, for our examples it will suffice to restrict even further and consider the 2-subcategory $\mscrZ(M)$ of $\mscrA(M)$ containing only single 1-morphism, namely the identity automorphism of $M$. It follows from the Definition \ref{def:2aut} that $\mscrZ(M)$ only depends on the center of $M$, hence we can as well assume from now on that $M$ is commutative. We then obtain

\begin{definition}
For a commutative monoid $M$, the simplicial set $K(M,2)$ is the Duskin nerve $\nerve\mscrZ(M)$ of the 2-category $\mscrZ(M)$. Thus its $n$-simplices are families $(a_{ijk})_{0\le i<j<k\le n}$ of elements of $M$ satisfying the equalities
\[
a_{ikl}a_{ijk}=a_{ijl}a_{jkl}
\]
for all $0\le i<j<k<l\le n$.
\end{definition}

Applying Proposition \ref{prop:catbc} to this particular case we obtain
\begin{corollary}
For a commutative monoid $M$, the simplicial set $K(M,2)$ satisfies $\BC_{pq}[n]$ ($0\le p<q\le n$, $n>1$) iff for any $a_{ijk}\in M$ with $0\le i<j<k\le n$, $\{p,q\}\nsubset\{i,j,k\}$ that for any $0\le i<j<k<l\le n$, $\{p,q\}\nsubset\{i,j,k,l\}$ satisfy
\[
a_{ikl}a_{ijk}=a_{ijl}a_{jkl}
\]
there exist elements $x_{ipq}$, $y_{pjq}$, $z_{pqk}$ of $M$ with $0\le i<p$, $p<j<q$, $q<k\le n$ satisfying
\begin{enumerate}
\item $x_{ipq}a_{ijp}=a_{ijq}x_{jpq}$ for all $0\le i<j<p$;
\item $y_{pkq}a_{pjk}=y_{pjq}a_{jkq}$ for all $p<j<k<q$;
\item $a_{pkl}z_{pqk}=z_{pql}a_{qkl}$ for all $q<k<l\le n$;
\item $x_{ipq}y_{pjq}=a_{ijq}a_{ipj}$ for all $0\le i<p<j<q$;
\item $a_{iqk}x_{ipq}=a_{ipk}z_{pqk}$ for all $0\le i<p<q<k\le n$;
\item $z_{pqk}y_{pjq}=a_{pjk}a_{jqk}$ for all $p<j<q<k\le n$.
\end{enumerate}
\end{corollary}

Using this description we may now give examples showing that not every $K(M,2)$ is an SS-complex and that there is one which is an SS-complex but not a Kan complex.

\begin{example}
Let $M$ be the monoid $\{0,1\}$ under multiplication. To test the condition $\BC_{0,3}[5]$ consider the family $a_{ijk}$ with $0\le i<j<k\le5$ with $\{0,3\}\nsubset\{i,j,k\}$ where $a_{015}=a_{024}=a_{135}=a_{234}=1$ and all the remaining $a_{ijk}$ zero. Then it is easy to check that in the equalities to be satisfied by these $a_{ijk}$ no more than one of the terms is equal to $1$, so that they are all satisfied. Then for $\BC_{0,3}[5]$ to hold there must in particular exist solutions to
\begin{align*}
z_{035}y_{013}&=1\\
z_{034}y_{023}&=1\\
z_{034}y_{013}&=0\\
z_{035}y_{023}&=0.
\end{align*}
To satisfy the first two of these equations one must have $z_{035}=y_{013}=z_{034}=y_{023}=1$, and then the remaining two equations cannot be satisfied. Thus $K(M,2)$ is not an SS-complex.
\end{example}

\begin{example}
Let now $\BbbZ_+$ be the monoid of nonnegative integers under addition. Then $K(\BbbZ_+,2)$ is not a Kan complex by Proposition \ref{prop:kanerve}.
Let us show that it is an SS-complex.
    
Note that $\BbbZ_+$ embeds in the group of integers $\BbbZ$, and $K(\BbbZ,2)$ is a Kan complex, hence an SS-complex by Proposition \ref{prop:kantobc}. To check the condition $\BC_{pq}[n]$, given $a_{ijk}\in\BbbZ_+$ for $\{p,q\}\nsubset\{i,j,k\}$ satisfying the required equalities, we can find the needed elements $x_{ipq}$, $y_{pjq}$, $z_{pqk}$ satisfying equations (1) -...- (6) in $\BbbZ$. Then note that the equations (4) imply $y_{pjq}\ge-x_{ipq}$ and the equations (6) imply that $y_{pjq}\ge-z_{pqk}$. Hence there is an $A\in\BbbZ$ such that $A\ge-x_{ipq}$, $A\ge-z_{pqk}$ and $A\le y_{pjq}$. In other words the numbers
\begin{align*}
\tilde x_{ipq}&=x_{ipq}+A\\
\tilde y_{pjq}&=y_{pjq}-A\\
\tilde z_{pqk}&=z_{pqk}+A
\end{align*}
are all in $\BbbZ_+$. And it is easy to check that these also satisfy equations (1) -...- (6).
\end{example} 

\section*{Remaining questions}

Is there a nontrivial model structure on simplicial sets with fibrations including all Skvortsov-Shehtman fibrations?

For a commutative monoid $M$, is the condition that $K(M,2)$ is an SS-complex equivalent to embeddability of $M$ into a group, i.~e. cancellability of $M$? If not, what is it?

More generally, can one characterize monoids $M$ for which $\nerve\mscrA(M)$ from Definition \ref{def:2aut} is an SS-complex?

\end{document}